\documentclass{amsart}

\newcommand{\IR}{\mathbb R}
\newcommand{\w}{\omega}
\newcommand{\pr}{\mathrm{pr}}
\newcommand{\e}{\varepsilon}

\newtheorem{theorem}{Theorem}
\newtheorem{lemma}{Lemma}
\newtheorem{claim}{Claim}
\newtheorem{problem}{Problem}

\title[Haar null $F_\sigma$-set]{A Polish group containing a Haar null $F_\sigma$-subgroup\\ that cannot be enlarged to a Haar null $G_\delta$-set}
\author{Taras Banakh}

\begin{document}
 \begin{abstract}
Answering a question of Elekes and Vidny\'anszky, we construct a Polish meta-abelian group $H$ and a subgroup $F\subset H$, which is a Haar null $F_\sigma$-set in $H$ that cannot be enlarged to a Haar null $G_\delta$-set.
\end{abstract}
\subjclass{22A10; 28C10}
\address{Ivan Franko National University of Lviv (Ukraine) and Jan Kochanowski University in Kielce (Poland)}
\email{t.o.banakh@gmail.com}
\maketitle

A Borel subset $B$ of a Polish group $H$ is called {\em Haar null} if there exists a $\sigma$-additive Borel probability measure $\mu$ on $H$ such that $\mu(xBy)=0$. It is well-known \cite{Chris} that a Borel subset of a Polish locally compact group $H$ is Haar null if and only if $B$ has Haar measure zero if and only if $B$ is contained in a Haar null $G_\delta$-subset of $H$. As was shown by Elekes and Vidny\'anszky \cite{EV}, for non-locally compact Polish groups the latter result is not true: each non-locally compact Polish abelian group $H$ contains a Borel Haar null subset $B\subset H$ which is not contained in a Haar null $G_\delta$-set in $H$. However, the construction of the Borel set $B$ exploited in \cite{EV} does not allow to evaluate the Borel complexity of $B$. Because of that, Elekes and Vidny\'anszky asked in \cite{EV} wherther each Haar null $F_\sigma$-subset $B$ of a Polish (Abelian) group $H$ can be enlarged to a Haar null $G_\delta$-set. In this section we partly answer this problem presenting an example of a Polish meta-Abelian group $H$ and a Haar null $F_\sigma$-subgroup $F\subset H$ that cannot be enlarged to a Haar null $G_\delta$-set.

A topological group $H$ is {\em meta-abelian} if $H$ contains a closed normal abelian subgroup $A\subset H$ such that the quotient group $H/A$ is Abelian.

We define a subset $B$ of topological group $H$ to be {\em thick} if for every compact subset $K\subset H$ there are points $x,y\in H$ such that $xKy\subset B$. It is easy to see that a thick Borel subset of Polish groups cannot be Haar null.

\begin{theorem}\label{t:main} There exists a Polish meta-abelian group $H$ containing a subgroup $F\subset H$ such that $F$ is a Haar null $F_\sigma$-set in $H$ but every $G_\delta$-set $G\subset H$ containing $B$ is thick and hence is not Haar null in $H$.
\end{theorem}

\begin{proof} Observe that the countable power $R=\IR^\w$ of the real lines has the structure of a unital topological ring and the dense $G_\delta$-set $(\IR\setminus\{0\})^\w$ coincides with the set $R^*$ of invertible elements of the ring $R$. On the product $H=R\times R^*$ consider the binary
operation $\star:H\times H\to H$ defined by the formula
$$(x,a)\star (y,b)=(x+ay,an)\mbox{ \ \ for \ \ } (x,a),(y,b)\in H=R\times R^*.$$
The Polish space $H=R\times R^*$ endowed with this binary operation is a Polish group called the {\em semidirect product} of $R$ and $R^*$.

In the topological ring $R=\IR^\w$ consider the $F_\sigma$-subset $$R_0=\{(x_n)_{n\in\w}\in\IR^\w:\exists n\in\w\;\;\forall m\ge n\;\;x_m=0\}$$ consisting of all eventually zero sequences. It follows that $R_0$ is a subring of $R$ and $F=R_0\times R^*\subset H$ is a subgroup of the Polish group $H$. We claim that the subgroup $F$ has the desired properties.

\begin{lemma} The subset $F$ is Haar null in $H$.
\end{lemma}

\begin{proof} Taking into account that $R_0$ is a Borel non-open subgroup of the Polish Abelian group $R$, we can apply a classical result of Christensen \cite{Chris} and conclude that $R_0$ is Haar null in $R$. This allows us to find a probability measure $\mu$ on $R$ such that $\mu(x+R_0)=0$ for all $x\in R$. Identifying $R$ with the normal subgroup $R\times\{1\}$ of $H$, we can consider the measure $\mu$ as a mesure on $H$. Then for every elements $(x,a),(y,b)\in H$ we get $$\big((x,a)\star F\star(y,b)\big)\cap \big(R\times\{1\}\big)=(x,a)\star \big(R_0\times\{a^{-1}b^{-1}\}\big)\star (y,b)=(x+a\cdot R_0+a^{-1}b^{-1}y,1)$$and hence
$$\mu\big((a,x)\star F\star (y,b)\big)=\mu(x+a\cdot R_0+a^{-1}b^{-1}y)=\mu(x+a^{-1}b^{-1}y+R_0)=0$$by the choice of $\mu$. So, the measure $\mu$ witnesses that the set $F=R_0\times R^*$ is Haar null in $H$.
\end{proof}

\begin{lemma} Every $G_\delta$-set $G\subset H$ containing $F$ is thick.
\end{lemma}

\begin{proof}
Given a $G_\delta$-set $G\subset H$ containing $F$, consider its complement $H\setminus G$ and write it as a countable union $H\setminus G=\bigcup_{k\in\w}E_k$ of an increasing sequence $(E_k)_{k\in\w}$ of closed sets in $H$. To prove that $G$ is thick in $H$, it suffices for every compact subset $K\subset H$ to find an element $h\in H$ such that $hKh^{-1}\subset G$. Given a compact set $K\subset H$ choose  compact sets $C\subset R$ and $C^*\subset R^*$ such that $K\subset C\times C^*$. Observe that for every element $a\in R^*$ and the element $h=(0,a)\in H$ we get $h^{-1}=(0,a^{-1})$ and hence
$$hKh^{-1}\subset (0,a)\star (C\times C^*)\star(0,a^{-1})\subset aC\times C^*.$$
Now it suffices to find an element $a\in R^*$ such that $(aC\times C^*)\cap E_k=\emptyset$ for all $k\in\w$. The compactness of the set $C^*$ guarantees that the projection $\pr:R\times C^*\to R$, $\pr:(x,y)\mapsto x$, is closed and for every $k\in\w$ the set $P_k=\pr(E_k\cap(R\times C^*))$ is closed in $R$ and does not intersect the $F_\sigma$-subgroup $R_0$. Observe that the set $R^*_k=\{a\in R^*:aC\cap P_k=\emptyset\}$ is contained in $\{a\in R^*:(aC\times C^*)\cap E_n\ne\emptyset\}$.
The compactness of the set $C$ implies that the set $R^*_n$ is open in $R^*$.

\begin{claim}\label{cl1} For every $k\in\w$ the set $R^*_k$ is dense in $R^*$.
\end{claim}

\begin{proof} Given any element $a=(a_n)_{n\in\w}\in R=\IR^\w$ and any neighborhood $O_a\subset R$,  find $m\in\w$ such that the closed subspace $\{(b_n)_{n\in\w}\in\IR^\w:\forall n<m\;\;b_n=a_n\}$ is contained in $O_a$. Find a sequence $(C_n)_{n\in\w}$ of compact subsets of the real line such that $C\subset \prod_{n\in\w}C_n$. Consider the compact space $\Pi=\prod_{n<k}a_kC_k$ and observe that the projection $\pi:\Pi\times\IR^{\w\setminus k}\to\IR^{\w\setminus k}$, $\pi:(x,y)\mapsto y$, is a closed map. This implies that the set $\tilde P_k=\pi_k\big((\Pi\times\IR^{\w\setminus k})\cap P_k\big)$ is closed in $\IR^{\w\setminus k}$ and does not contain the constant zero function $z\in\IR^{\w\setminus k}$. By the compactness of the product $C_{\ge k}=\prod_{n\ge k}C_n\subset \IR^{\w\setminus k}$ and the continuity of multiplication in the topological ring $\IR^{\w\setminus k}$, there exists an element $\e\in (\IR\setminus\{0\})^{\w\setminus k}$ so close to zero that $(\e\cdot C_{\ge k})\cap \tilde P_k=\emptyset$. Then for the element $b=(b_n)_{n\in\w}\in O_a$ defined by $b_n=a_n$ for $n<k$ and $b_n=\e_n$ for $n\ge k$
we get $$(b\cdot C)\cap P_n\subset \Big(\prod_{n\in\w}b_nC_n\Big)\cap P_n\subset (\prod_{n<k}a_nC_n)\times \Big(\big(\prod_{n\ge k}\e_nC_n\big)\cap\tilde P_k)=\emptyset$$and hence
 $b\in O_a\cap R^*_k$.
\end{proof}

Claim~\ref{cl1} combined with the Baire Theorem guarantees that $\bigcap_{k\in\w}R_k^*$ is a dense $G_\delta$-set in the Polish space $R^*$. Then we can take any point $a\in \bigcap_{k\in\w}R_k^*$ and conclude that $(aC\times C^*)\cap \bigcup_{k\in\w}E_k=\emptyset$ and hence for the element $h=(0,a)$ we get $hKh^{-1}\subset G$.
\end{proof}
\end{proof}

However, Theorem~\ref{t:main} gives no answer to the following two problems posed by Elekes and Vidny\'anszky in \cite{EV}.

\begin{problem} Is each Haar null $F_\sigma$-subset of an uncountable Polish Abelian group $G$ contained in a Haar null $G_\delta$-subset of $G$?
\end{problem}

\begin{problem}\label{prob2} Is each countable subset of an uncountable Polish group $G$ contained in a Haar null $G_\delta$-subset of $G$?
\end{problem}

By Remark 5.3 \cite{EV} the answer to Problem~\ref{prob2} is affirmative for Polish abelian groups.


\begin{thebibliography}{}

\bibitem{Chris} J.~Christensen, {\em On sets of Haar measure zero in abelian Polish groups}, Israel J. Math. 13 (1972), 255--260 (1973).

\bibitem{EV} M.~Elekes,  Z.~Vidny\'anszky, {\em Haar null sets without $G_\delta$-hulls},  Israel J. Math. 209 (2015) 199--214.

\end{thebibliography}
\end{document}